\documentclass{amsart}
\usepackage{amssymb}

\theoremstyle{plain}
\newtheorem{thm}{Theorem}[section]
\newtheorem{lemma}[thm]{Lemma}

\newtheorem{prop}[thm]{Proposition}

\theoremstyle{definition}

\newtheorem{defn}[thm]{Definition}
\newtheorem{defin}[thm]{Definition}

\newtheorem*{ack}{Acknowledgements}

\newcommand{\bSigma}{\mathbf{\Sigma}}

\newcommand{\hgt}{\mathrm{ht}}

\newcommand{\ZFC}{\textrm{ZFC}}

\newcommand{\MA}{\textrm{MA}}

\newcommand{\dom}{\textrm{dom}}

\newcommand{\pol}{\mathrm{pol}}
\newcommand{\CB}{\mathrm{CB}}
\newcommand{\Nrm}{\textrm{Nrm}}

\renewcommand{\P}{\mathbb{P}}%
\providecommand{\Q}{\mathbb{Q}}%
\providecommand{\bdd}{{\rm bdd}}

\providecommand{\Par}{{\mathfrak{par}}}
\providecommand{\poly}{{\rm poly}}

\providecommand{\Hom}{\mathfrak{hom}}

\begin{document}

\author[Palumbo]{Justin Palumbo}
\address{Justin Palumbo, Department of Mathematics,
                    University of California at Los Angeles,
                    Los Angeles, California}
\email{justinpa@math.ucla.edu}
%\subjclass[2000]{03E25, 03E35}\keywords{set theory, forcing}

%\begin{abstract}
%yo Finally, we give characterizations of 
%\end{abstract}

\title{Comparisons of polychromatic and monochromatic Ramsey theory}

\begin{abstract}
We compare the strength of polychromatic and monochromatic Ramsey theory in several set-theoretic domains. We show that the rainbow Ramsey theorem does not follow from ZF, nor does the rainbow Ramsey theorem imply Ramsey's theorem over ZF. Extending the classical result of Erd{\"o}s and Rado we show that the axiom of choice precludes the natural infinite exponent partition relations for polychromatic Ramsey theory. We introduce rainbow Ramsey ultrafilters, a polychromatic analogue of the usual Ramsey ultrafilters. We investigate the relationship of rainbow Ramsey ultrafilters with various special classes of ultrafilters, showing for example that every rainbow Ramsey ultrafilter is nowhere dense but rainbow Ramsey ultrafilters need not be rapid. This entails comparison of the polychromatic and monochromatic Ramsey theorems as combinatorial principles on $\omega$. Finally we give new characterizations of the bounding and dominating numbers and the covering and uniformity numbers of the meager ideal which are in the spirit of polychromatic Ramsey theory.
\end{abstract}

\maketitle

\section{Introduction}

In this article we investigate the relative strengths of monochromatic and polychromatic Ramsey theory in a variety of settings. Recall that in the usual monochromatic Ramsey theory one is given a coloring $\chi:[X]^n\rightarrow C$ and seeks a set $Y\subseteq X$ which is {\em monochromatic} for $\chi$. This means that there is a single color which all elements of $[Y]^n$ receive. In the polychromatic Ramsey theory we instead seek a set $Y\subseteq X$ which is {\em polychromatic} for $\chi$. This means that each member of $[Y]^n$ receives a different color. (Polychromatic Ramsey theory also goes by the name rainbow Ramsey theory; a polychromatic set might be called a rainbow).

In order to be able to find monochromatic or polychromatic sets we must put some restriction on the colorings under consideration. In monochromatic Ramsey theory the appropriate restriction is to insist that the set of colors be finite; in the polychromatic theory we insist that each color gets used a bounded, finite number of times. For $k\in\omega$ we will say that the coloring $\chi:[X]^n\rightarrow C$ is $k$-bounded if $|\chi^{-1}[c]|\leq k$ for each $c\in C$.

The following trick due to Fred Galvin shows that whenever positive results in the monochromatic theory hold so too will their polychromatic analogue. Suppose we are given $\chi:[X]^n\rightarrow C$ a $k$-bounded coloring. For each $c\in C$ fix an enumeration of $\chi^{-1}[c]$, and form the dual coloring $\chi^*:[X]^n\rightarrow k$ by letting $\chi(a)=i$ exactly when $a$ is the $i$th element in the enumeration of its color class. It is easy to see that $Y\subseteq X$ is polychromatic for $\chi$ whenever $Y$ is monochromatic for $\chi^*$.

Several situations in which the polychromatic theory is strictly weaker than the monochromatic theory are already well-known. In the finite setting it has been shown that the classical Ramsey number $R_n$ grows much more quickly than its polychromatic counterpart (\cite{Hell}, \cite{HellAgain}). In the context of reverse mathematics, Csima and Mileti showed \cite{Csima} that the rainbow Ramsey theorem does not imply Ramsey's theorem over $\textrm{RCA}_0$ even though $\textrm{RCA}_0$ is sufficient to prove that Ramsey's theorem implies the rainbow Ramsey theorem. Sierpinksi showed  (in ZFC) that there are $2$-colorings of $[\omega_1]^2$ with no monochromatic subset of size $\omega_1$, yet Todor{\v{c}}evi{\'c} \cite{Tod} and Abraham, Cummings and Smyth \cite{ACS} independently showed that under PFA one may always find polychromatic subsets of size $\omega_1$ for $2$-bounded colorings on $[\omega_1]^2$.

Our contributions are the following. In section \ref{RRTnAC} we view the rainbow Ramsey theorem as a choice principle. We will prove that some choice is needed to prove the rainbow Ramsey theorem, and that there are models of ZF where the rainbow Ramsey theorem holds yet Ramsey's theorem fails. In section \ref{RRTnIER} we will show that the axiom of choice forbids infinite exponent partition relations for the polychromatic Ramsey theory just as it does for the monochromatic theory.

In section \ref{RRTnUF} we investigate the (countable) combinatorial power of the rainbow Ramsey theorem; to accomplish this we introduce rainbow Ramsey ultrafilters, a polychromatic analogue of the classical Ramsey ultrafilters. Every Ramsey ultrafilter is a rainbow Ramsey ultrafilter, yet consistently there are rainbow Ramsey ultrafilters which are not Ramsey. Thus in the context of ultrafilters the polychromatic theory is weaker. We will investigate the relationship between rainbow Ramsey ultrafilters and other well-known types of special ultrafilters which encapsulate various combinatorial principles on $\omega$. Ramsey's theorem is sufficiently strong as a combinatorial to principle so that any Ramsey ultrafilter falls into one of these special types; this is not so for the rainbow Ramsey theorem.

Constructing ultrafilters which are rainbow Ramsey but fail to have some other property requires building polychromatic sets with various special properties, properties for which one cannot generally find monochromatic sets. For example, we will (assuming $\MA$) construct a rainbow Ramsey ultrafilter which is not rapid. To do this we must be able to build polychromatic sets whose enumerating functions do not grow too fast; the monochromatic theory is strong enough to enforce fast growth while the polychromatic theory is not. In a similar vein we will construct a rainbow Ramsey ultrafilter which is not discrete. This requires building polychromatic subsets of $\Q$ which have high Cantor-Bendixson rank, something not in general possible in the monochromatic theory. However we will prove that every rainbow Ramsey ultrafilter is nowhere dense. One aspect of the proof involves showing that there are $2$-bounded colorings for which polychromatic sets are necessarily nowhere dense. We also show that there may exist weakly selective ultrafilters which are not rainbow Ramsey.

Finally, in section \ref{RRTnCC} we give several cardinal characteristics of the continuum new characterizations in the spirit of polychromatic Ramsey theory.

\begin{ack}
The author would like to thank Itay Neeman for many helpful discussions, and for helping to simplify the exposition of the material in Section \ref{discreteUlt}.
\end{ack}

\section{Polychromatic Ramsey theory and the axiom of choice} \label{RRTnAC}

In this section we investigate polychromatic Ramsey theory in the absence of the axiom of choice. The standard reference for the basics of building models for the failure of choice is the text \cite{Je}, whose notation and terminology we follow closely. 

The original result in Ramsey theory may be stated as follows. 
 
\begin{thm}[Ramsey's Theorem]\label{RamseysThm}
Let $X$ be an infinite set and let $\chi:[X]^2\rightarrow 2$. Then there is an infinite $Y\subseteq X$ which is monochromatic for $\chi$.
\end{thm}

Applying Galvin's trick one gets the following polychromatic analogue.

\begin{thm}[The rainbow Ramsey theorem]\label{RainbowRamseyThm}
Let $X$ be an infinite set and let $\chi:[X]^2\rightarrow C$ be a 2-bounded coloring. Then there is an infinite $Y\subseteq X$ which is polychromatic for $\chi$.
\end{thm}
 
The proof of Ramsey's theorem as stated here uses the fact that every infinite set has a countably infinite subset, while Galvin's trick requires the existence of a choice function on sets of pairs. Kleinberg \cite{Kl} proved that some amount of choice is necessary to prove Ramsey's theorem. We begin this section by observing that this is also true of the rainbow Ramsey theorem. 

\begin{thm}
There is a model of ZF in which the rainbow Ramsey theorem does not hold.
\end{thm}

\begin{proof}
We use the permutation model $M$ referred to in \cite{Je} as the second Fraenkel model. While technically speaking permutation models only yield independence results for ZFA (set theory with atoms), the Jech-Sochor theorem (Theorem 6.1 of \cite{Je}) can be applied to yield the ZF result.

Recall that the model $M$ is obtained as follows. Let $A=\bigcup P_n$ where $P_n=\{a_n,b_n\}$, and $\mathcal{G}$ is the group of all permutations $\pi$ of $A$ such that $\pi(\{a_n,b_n\})=\{a_n,b_n\}$. We obtain $M$ using $\mathcal{G}$ and the ideal $\mathcal{I}$ of finite supports.

Let $\chi$ in $M$ be a 2-bounded coloring of $[A]^2$ which gives the pair $\{a_i,b_j\}$ the same color as $\{b_i,a_j\}$, and the pair $\{a_i,a_j\}$ the same color as the pair $\{b_i,b_j\}$. Specifically we may take $\chi$ to be defined by $\chi(\{a_i,b_j\})=\{\{a_i,b_j\},\{a_j,b_i\}\})$ and $\chi(\{a_i,a_j\})=\chi(\{b_i,b_j\})=\{\{a_i,a_j\},\{b_i,b_j\}\}$. It is not hard to see that $\chi$ is invariant under permutations $\pi\in\mathcal{G}$ and hence that $\chi$ belongs to $M$.

There is no infinite set in $M$ which is polychromatic for $\chi$. This is because any infinite $B\subseteq A$ belonging to $M$ must contain infinitely many pairs $\{a_i,b_i\}$.
%otherwise let X be such a set, and E its finite support. for example, get a_i,b_j in X\setminus E with i not equal to j
%then if pi fixes E and swaps a_i with b_i and b_j with a_j then b_i,a_j is also in X and gets the same color. 
\end{proof}

Let $N$ be the basic Cohen model of the failure of the axiom of choice, as described in section 5.3 of \cite{Je}. In that model the Boolean prime ideal theorem holds, every set can be linearly ordered and every collection of well-ordered sets has a choice function. Blass \cite{Bl2} proved that Ramsey's theorem fails in $N$. Thus our next result shows that the rainbow Ramsey theorem is considerably weaker than Ramsey's theorem as a choice principle. Our argument is very much inspired by Blass's argument in \cite{Bl2} that the basic Fraenkel model satisfies Ramsey's theorem.

\begin{thm}\label{RRTdontimplyRT}
The rainbow Ramsey theorem holds in $N$.
\end{thm}

As in \cite{Je} we take $A=\{x_n:n\in\omega\}$ to be the canonical set of Cohen reals in $N$. We start by showing that the rainbow Ramsey theorem holds on $A$.

\begin{lemma}
Say $Y\subseteq A$ is infinite, $Y\in N$. If $\chi:[Y]^2\rightarrow C$ is a two-bounded coloring in $N$, then there is in $N$ an infinite set $X\subseteq Y$ on which $\chi$ is polychromatic.
\end{lemma}

\begin{proof}
Let $\dot{\chi}$, $\dot{Y}$ be hereditarily symmetric names for the corresponding objects. There is a finite $e\subseteq\omega$ such that $\textrm{fix}(e)\subseteq\textrm{sym}(\dot{\chi})$. We claim that $Y\setminus\{x_n:n\in e\}$ is polychromatic.

Suppose otherwise for contradiction. There are two very similar cases to consider; we will derive a contradiction from the situation where for some distinct $i_0,i_1,i_2$ not in $e$ we have $\chi(\{x_{i_0},x_{i_2}\})=\chi(\{x_{i_1},x_{i_2}\})$. Let $p\in\mathbb{P}$ forcing $\dot{\chi}$ to be $2$-bounded and with $$p\Vdash\dot{\chi}(\{\dot{x}_{i_0},\dot{x}_{i_2}\})=\dot{\chi}(\{\dot{x}_{i_1},\dot{x}_{i_2}\}).$$ We may assume without loss of generality that $e\cup\{i_0,i_1,i_2\}\subseteq\dom(p)$. Let $\pi$ be a permutation which fixes each member of $e$ as well as $i_0$ and $i_2$, and for which $\pi(i_1)=k$ where $k$ is not in the domain of $p$. Then $p,\pi(p)$ are compatible, and $\pi(p)\Vdash\dot{\chi}(\{\dot{x}_{i_0},\dot{x}_{i_2}\})=\dot{\chi}(\{\dot{x}_k,\dot{x}_{i_2}\})$. But then if $q$ is a common strengthening of $p$ and $\pi(p)$ we have that $$q\Vdash \dot{\chi}(\{\dot{x}_{i_0},\dot{x}_{i_2}\})=\dot{\chi}(\{\dot{x}_k,\dot{x}_{i_2}\})=\dot{\chi}(\{\dot{x}_{i_1},\dot{x}_{i_2}\})$$ which violates $\chi$ being forced to be two-bounded.

The case where there exist distinct $i_0,i_1,i_2,i_3$ not in $e$ with $\chi(\{x_{i_0},x_{i_1}\})=\chi(\{x_{i_2},x_{i_3}\})$ is similarly handled.
\end{proof}

To prove that the rainbow Ramsey theorem holds in $N$ we must show that if $X\in N$ is an infinite set and $\chi:[X]^2\rightarrow C$ is a $2$-bounded coloring in $N$, then $N$ contains an infinite polychromatic subset of $X$. If $X$ happens to be well-orderable in $N$, there is no difficulty since the usual proof of the rainbow Ramsey theorem will go through. We thus only have to worry about sets in $N$ which cannot be well-ordered. Theorem \ref{RRTdontimplyRT} will be proven once we establish the following. %dichotomy

\begin{lemma}
If $B\in N$ is a non-wellorderable set, then $N$ contains a bijection of $B$ with an infinite subset of $A$.
\end{lemma}

\begin{proof}
We take advantage of the theory of least supports. Our notation will match that of chapter 5 of \cite{Je}. Let $x\in N$, and $E\subseteq A$ with $E$ finite. Write $E=\{x_{i_0},\ldots x_{i_k}\}$. We say that $E$ supports $x$ and write $\Delta(E,x)$ if there is some hereditarily symmetric name $\dot{x}$ with $\dot{x}[G]=x$ such that $\{i_0,\ldots i_k\}$ supports $\dot{x}$. The class relation $\Delta$ is definable in $N$.

We first claim that if $B\in N$ and there is some single $E$ such that $\Delta(E,x)$ holds for all $x\in B$, then $B$ can be well-ordered in $N$. Say $E=\{x_{i_0},\ldots x_{i_k}\}$. By the axiom of replacement applied in $N$ there is an ordinal $\alpha$ so that for every $x\in B$ there is a hereditarily symmetric name $\dot{x}$ in $V_\alpha$ with support $e=\{i_0,\ldots i_k\}$ and $\dot{x}[G]=x$. Let $C=\{\tau[G]:\textrm{$\tau$ is a hereditarily symmetric name in $V_\alpha$ with support $e$}\}$. Then $C$ belongs to $N$ and $B\subseteq C$. Furthermore, $C$ can be well-ordered in $N$ since all the relevant names are supported by $e$. Thus $B$ can be well-ordered in $N$.

Recall now that every $x\in N$ has least support; that is, there is some $E_0$ with $\Delta(x,E_0)$ and such that $\Delta(x,E)$ implies $E_0\subseteq E$. Suppose $B$ is some non-wellorderable set with least support $E_0$, witnessed by the hereditarily symmetric name $\dot{B}$. By the above paragraph there is some $x\in B$ for which $\Delta(x,E_0)$ does not hold. Write the least support of $x$ as $E_1\cup\{x_k\}$ where $x_k$ does not belong to $E_0\cup E_1$. Let $\dot{x}$ be a hereditarily symmetric name for $x$ with support $E_1\cup\{x_k\}$. Enumerate $E_0=\{x_{i_0},\ldots x_{i_{l_1}}\}$ and $E_1=\{x_{j_0},\ldots x_{j_{l_2}}\}$. Let $p\in G$ with $$p\Vdash\dot{x}\in\dot{B} \textrm{ and $\dot{E}_1\cup\{\dot{x}_k\}$ is the least support of $\dot{x}$.}$$

Let $e_0=\{i_0,\ldots i_{l_1}\}$, let $e_1=\{j_0,\ldots j_{l_2}\}$ and let $e=e_0\cup e_1$. Consider the name $$\sigma=\{\langle\langle\pi(\dot{x}_k),\pi(\dot{x}) \rangle  ,\pi(p)   \rangle     :\pi\textrm{ fixes }e \}.$$ Then $f=\sigma[G]$ belongs $N$ since $\sigma$ is supported by $e$. We have four things to prove about $f$.

First, we check that $f$ is a function. If $\pi_1(\dot{x}_k)[G]=\pi_2(\dot{x}_k)[G]$ then $\pi_1(k)=\pi_2(k)$ and so $\pi_1^{-1}\pi_2$ fixes $e\cup\{k\}$ and thus $\dot{x}$. Then $\pi_1(\dot{x})=\pi_2(\dot{x})$.

Secondly we note that the range of $f$ is a subset of $X$. Any member of the range of $f$ has the form $\pi(\dot{x})[G]$ for some $\pi$ fixing $e$ with $\pi(p)\in G$. Since $\pi$ fixes $e$, $\pi(\dot{B})=\dot{B}$ and so $\pi(p)\Vdash\pi(\dot{x})\in\dot{B}$.

Next we observe that the domain of $f$ is an infinite subset of $A$. This is because $x_i$ belongs to the domain of $f$ whenever $\pi(p)\in G$ for some $\pi$ mapping $k$ to $i$. Since $G$ is generic this happens for infinitely many $i$.

Finally we claim that $f$ is injective. Otherwise we would have $\pi_1(\dot{x})[G]=\pi_2(\dot{x})[G]$ for some $\pi_1,\pi_2$ with $\pi_1(k)=k_1$, $\pi_2(k)=k_2$, $k_1\not=k_2$ and both $\pi_1(p)$ and $\pi_2(p)$ belonging to $G$. Let $q\leq\pi_1(p),\pi_2(p)$ with $q$ in $G$ and such that $q\Vdash\pi_1(\dot{x})=\pi_2(\dot{x})$. Now $\pi_1(\dot{x})$ and $\pi_2(\dot{x})$ have supports $e_1\cup\{k_1\}$ and $e_1\cup\{k_2\}$ respectively. Hence by \cite{Je} Lemma 5.23 there is some $\dot{z}$ with support $e_1$ and $q\Vdash\pi_1(\dot{x})=\pi_2(\dot{x})=\dot{z}$. Since $q$ extends $\pi_1(p)$ we have $$q\Vdash\textrm{$\dot{E}_1\cup\{\dot{x}_{k_1}\}$ is the least support of $\pi_1(\dot{x})$  }.$$ On the other hand, $q$ belongs to $G$ and $\pi_1(\dot{x})[G]=\dot{z}[G]$ and $\dot{z}$ has support strictly contained in $e_1\cup\{k_1\}$. Contradiction.
\end{proof}

As we mentioned, Galvin's trick also requires a small amount of choice. We do not know of a model of ZF where Ramsey's theorem holds but the rainbow Ramsey theorem fails.

\section{Polychromatic Ramsey theory and infinite exponent partition relations} \label{RRTnIER}

A result of Erd{\"o}s and Rado says that under the axiom of choice Ramsey's theorem fails for infinite exponent partitions (Proposition 7.1 of \cite{Kanamori}). Specifically, for any infinite cardinal $\kappa$ there is a $2$-coloring of the countable subsets of $\kappa$ so that no infinite subset of $\kappa$ has all of its countable subsets receiving the same color. In this section we show that the axiom of choice also implies the failure of the rainbow Ramsey theorem for infinite exponent partitions. Using Galvin's trick we may view our result as a strengthening of the Erd{\"o}s and Rado result. The work in this section is joint with Anush Tserunyan.
 
\begin{thm} \label{IE}
Let $\kappa$ be an infinite cardinal. There is a $2$-bounded coloring $\chi:[\kappa]^\omega\rightarrow C$ so that whenever $X\in[\kappa]^\omega$ there are distinct $a,b\in [X]^\omega$ with $\chi(a)=\chi(b)$.
\end{thm}

To prove the theorem it is enough for us to establish the following.

\begin{lemma} \label{IElemma}
Let $\kappa$ be an infinite cardinal. There exists an injective map $f:[\kappa]^\omega\rightarrow[\kappa]^\omega$ so that for each $x$ we have that $f(x)$ is a proper subset of $x$.
\end{lemma}

Given the lemma, Theorem \ref{IE} is proven as follows. Let $f$ be as in Lemma \ref{IElemma}. We define $f_0$ and $f_1$ two injections from $[\kappa]^\omega$ into $[\kappa]^\omega$ with disjoint ranges so that $f_0(x)$ and $f_1(x)$ are both proper subsets of $x$ for each $x$. This can be done by looking at the orbits of $f$; that is, each collection $\{f^n(x):n\in\mathbb{Z}\}$. Because $f$ is injective, the orbits partition the range of $f$. Select an enumeration of each, and take $f_0$ and $f_1$ so that $f_0(x)$ is an even member of the orbit of $x$ while $f_1(x)$ is an odd member of the orbit of $x$. With $f_0$ and $f_1$ defined, we may define $\chi$ by setting $\chi(f_0(x))=\chi(f_1(x))$, and letting $\chi$ take distinct values on the other members of $[\kappa]^\omega$. Then $\chi$ is as desired.

Let us remark that since there are models of ZF where Ramsey's theorem holds for infinite exponent partition relations, by Galvin's trick there are models of ZF where the rainbow Ramsey theorem holds for infinite exponent partition relations. Thus this argument also shows that the axiom of choice is required to prove the existence of such injections $f_0$ and $f_1$.

We now make the observation that the lemma holds if $\kappa=\omega$.

\begin{prop} \label{IEprop}
There exists an injection $f:[\omega]^\omega\rightarrow[\omega]^\omega$ so that for each $x$ we have that $f(x)$ is a proper subset of $x$.
\end{prop}

\begin{proof}
Fix an enumeration of $[\omega]^\omega$ in ordertype $2^\omega$. We define $f$ by transfinite recursion. At stage $\alpha$, we define $f(x)$ where $x$ is the $\alpha$th member of $[\omega]^\omega$. Since $x$ has $2^\omega$ many proper subsets and since there are strictly less than $2^\omega$ values of $f$ which have been decided we may select a value for $f(x)$ not equal to any earlier decided value.
\end{proof}

\begin{proof}[Proof of Lemma \ref{IElemma}]
Fix $A=\{a_\alpha:\alpha<\lambda\}$ a maximal almost disjoint family of members of $[\kappa]^\omega$. That is, $a_\alpha\cap a_\beta$ is finite for distinct $\alpha$ and $\beta$ and for any $x\in[\kappa]^\omega$ there is some $\alpha$ such that $x\cap a_\alpha$ is infinite.

We construct $f$ as follows. For each $\alpha<\lambda$ let $f_\alpha:[a_\alpha]^\omega\rightarrow[a_\alpha]^\omega$ be as in Proposition \ref{IEprop}. Given $x$ in $[\kappa]^\omega$, take $\alpha$ least for which $x\cap a_\alpha$ is infinite and set $f(x)$ equal to $f_\alpha(x\cap a_\alpha)\cup (x\setminus a_\alpha)$.

We claim that $f$ is injective. Fix $x_0,x_1\in[\kappa]^\omega$ with $\alpha_0$ and $\alpha_1$ least so that $x_0\cap a_{\alpha_0}$ and $x_1\cap a_{\alpha_1}$ are infinite. Assume without loss of generality that $\alpha_0\leq\alpha_1$. Suppose $f(x_0)=f(x_1)$ so that $$f_{\alpha_0}(x_0\cap a_{\alpha_0})\cup (x_0\setminus a_{\alpha_0})=f_{\alpha_1}(x_1\cap a_{\alpha_1})\cup (x_1\setminus a_{\alpha_1}).$$ Note that $f_{\alpha_0}(x_0\cap a_{\alpha_0})$ is an infinite subset of $a_{\alpha_0}$ and $f_{\alpha_1}(x_1\cap a_{\alpha_1})$ is an infinite subset of $a_{\alpha_1}$. Consider the following two possibilities. If $f_{\alpha_0}(x_0\cap a_{\alpha_0})$ has infinite intersection with $f_{\alpha_1}(x_1\cap a_{\alpha_1})$, then $a_{\alpha_0}\cap a_{\alpha_1}$ is infinite so that $\alpha_0=\alpha_1$. The other possibility is that $f_{\alpha_0}(x_0\cap a_{\alpha_0})\cap x_1\setminus a_{\alpha_1}$ is infinite in which case $a_{\alpha_0}\cap x_1$ is infinite so that $\alpha_0$ is equal to $\alpha_1$ by minimality of $\alpha_1$. In either case we can conclude that $\alpha_0$ and $\alpha_1$ are equal to the same ordinal $\alpha$.

Thus $$f_\alpha(x_0\cap a_\alpha)\cup (x_0\setminus a_\alpha)=f_\alpha(x_1\cap a_\alpha)\cup (x_1\setminus a_\alpha).$$ Since $f_\alpha(x_0\cap a_\alpha)=f_\alpha(x_1\cap a_\alpha)$ we have $x_0\cap a_\alpha=x_1\cap a_\alpha$ by injectivity of $f_\alpha$. Because $x_0\setminus a_\alpha=x_1\setminus a_\alpha$ also holds we get $x_0=x_1$ as desired. 
\end{proof}

\section{Polychromatic Ramsey theory and ultrafilters on $\omega$} \label{RRTnUF}

We turn our attention now to monochromatic and polychromatic Ramsey theory in the context of ultrafilters on $\omega$. The following objects are central in the study of such ultrafilters.

\begin{defin}
A nonprincipal ultrafilter $\mathcal{U}$ is \emph{Ramsey} if for every coloring $\chi:[\omega]^2\rightarrow 2$ there is an $A\subseteq\omega$ belonging to $\mathcal{U}$ which is monochromatic for $\chi$.
\end{defin}

Ramsey ultrafilters are often called \emph{selective} ultrafilters in connection with the following characterization. An ultrafilter $\mathcal{U}$ is Ramsey exactly when given any partition of $\omega$ into countably many pieces $\bigcup_{n<\omega}A_n$ with each $A_n\not\in\mathcal{U}$ we may find $B\in\mathcal{U}$ such that $|A_n\cap B|\leq 1$ for each $n\in\omega$. Another salient characterization of Ramsey ultrafilters is that they are precisely nonprincipal ultrafilters which are minimal in the Rudin-Keisler ordering.

The existence of Ramsey ultrafilters is not provable in $\ZFC$. This was first established by Kunen \cite{Kunen}. Martin's Axiom ($\MA$) is sufficient to prove their existence; indeed $\MA$ is generally the context in which relationships between various classes of ultrafilters are studied. Such investigations have been pursued by Baumgartner \cite{Baumgartner}, Brendle \cite{Brendle} and others.

As an analogue to the Ramsey theoretic characterization of Ramsey ultrafilters, we present the following definition.

\begin{defin}
A nonprincipal ultrafilter $\mathcal{U}$ is \emph{rainbow Ramsey} if for every $2$-bounded coloring $\chi:[\omega]^2\rightarrow \omega$ there is
an $A\subseteq\omega$ belonging to $\mathcal{U}$ which is polychromatic for $\chi$.
\end{defin}

By Galvin's trick every Ramsey ultrafilter is a rainbow Ramsey ultrafilter. Assuming $\MA$ we will prove that the converse does not hold. We will also compare the notion  of rainbow Ramsey utrafilter to other notable classes of special ultrafilters on $\omega$. Let us introduce the special ultrafilters we will consider.

\begin{defn}
A nonprincipal ultrafilter $\mathcal{U}$ on $\omega$ is \emph{weakly selective} if whenever $\omega$ is partitioned into countably many pieces $\bigcup_{n<\omega}A_n$ with each $A_n\not\in\mathcal{U}$ we may find $B\in\mathcal{U}$ such that $A_n\cap B$ is finite for each $n\in\omega$.
\end{defn}

Weakly selective ultrafilters are also often referred to as P-points in connection with the fact that an ultrafilter is weakly selective exactly when for every countable family $\{B_n:n\in\omega\}$ of members of $\mathcal{U}$ there is some $B\in\mathcal{U}$ such that $B\subseteq^*B_n$ for each $n\in\omega$. (Here $\subseteq^*$ is the preorder of almost containment; $A\subseteq^* B$ means $A\setminus B$ is finite).

\begin{defn}
A nonprincipal ultrafilter $\mathcal{U}$ on $\omega$ is rapid if for every $f:\omega\rightarrow\omega$ there is some $A\in\mathcal{U}$ such that $f\leq^*e_A$. (Here $e_A$ is the function enumerating $A$ in increasing order, and $\leq^*$ is the preorder of eventual domination.)
\end{defn}

The next definition scheme is due to Baumgartner \cite{Baumgartner}. In this paper ideals will always contain all possible finite sets.

\begin{defn}
Let $\mathcal{I}$ be an ideal on some set $X$. We say that a nonprincipal ultrafilter on $\omega$ is an $\mathcal{I}$-ultrafilter if for every $f:\omega\rightarrow X$ there is some $A\in\mathcal{U}$ with $f(A)\in\mathcal{I}$.
\end{defn}

For example we could take $\mathcal{I}$ to be the nowhere dense subsets of $\mathbb{Q}$, or we could take $\mathcal{I}$ to be the discrete subsets of $\mathbb{Q}$; in these cases we have the notion of a nowhere dense ultrafilter and the notion of a discrete ultrafilter, respectively. Every Ramsey ultrafilter is rapid, and every Ramsey ultrafilter is weakly selective. Every weakly selective ultrafilter is discrete, and every discrete ultrafilter is nowhere dense.

We connect rainbow Ramsey ultrafilters to these classes as follows. Every rainbow Ramsey ultrafilter is nowhere dense, but (assuming $\MA$) there exist rainbow Ramsey ultrafilters which are not discrete as well as rainbow Ramsey ultrafilters which are not rapid. Shelah proved \cite{Shelah} that there are models of ZFC with no nowhere dense ultrafilters; this implication shows that the same is true of rainbow Ramsey ultrafilters. We will also show that $\MA$ implies the existence of a weakly selective ultrafilter which is not rainbow Ramsey. Together these results rule out the possibility of the concept of rainbow Ramsey ultrafilter being equivalent to any previously studied special class of ultrafilter.

In our constructions of ultrafilters which are rainbow Ramsey but lack some other property we will be interested in building polychromatic sets which are large in some sense. Let us describe some tools that will help us accomplish this. Fix a coloring $\chi:[\omega]^2\rightarrow\omega$. We assume throughout that $\chi$ is $2$-bounded. For $a,b\in\omega$ we will usually write $\chi(a,b)$ for $\chi(\{a,b\})$. If $X\subseteq\omega$ is finite and $a\in\omega$ we write $X<a$ to mean $\max(X)<a$. Similarly we will write $X<Z$ to mean $\max(X)<\min(Z)$.

\begin{defin} \label{normDef}
A set $A\subseteq\omega$ is {\em normal} if whenever $a_0<a_1$ and $b_0<b_1$ are elements of $A$ with $\chi(a_0,a_1)=\chi(b_0,b_1)$ then we necessarily have $a_1=b_1$.
\end{defin}

Generally our constructions of large polychromatic sets will entail first building large normal sets.

Suppose $X$ is a given finite polychromatic set. We define $E(X)$ by setting $$E(X)=\{a:X\cup a\textrm{ is polychromatic}\}.$$ We will sometimes write $E(X\cup x)$ as shorthand for $E(X\cup\{x\})$. The notation $A\subseteq^*B$ means $B\setminus A$ is finite.

\begin{prop} \label{union}
Suppose $A$ is normal. If $X\subseteq A$ is polychromatic, $|X|\leq n$ and $a_0<\ldots<a_n$ belong to $A\cap E(X)$ then $$A\cap E(X)\subseteq^*A\cap(E(X\cup a_0)\cup\ldots\cup E(X\cup a_n)).$$
\end{prop}

\begin{proof}
We are claiming that every member of $A\cap E(X)$ greater than $a_n$ belongs to $A\cap\bigcup_{i\leq n}E(X\cup a_i)$. Enumerate $X=\{x_0,\ldots x_{n-1}\}$. Suppose for contradiction that $z>a_n$ with $z\in A\cap E(X)$, but $z$ does not belong to any $A\cap E(X\cup a_i)$. For each such $i$, since $X\cup\{z\}$ is polychromatic, $X\cup\{a_i\}$ is polychromatic, $X\cup\{a_i,z\}$ is not polychromatic and $A$ is normal there must be some $j_i$ such that $\chi(a_i,z)=\chi(x_{j_i},z)$. There are $n+1$ possible $i$ while only $n$ possible $j_i$. By the pigeonhole principle there is some $j$ and $i_0<i_1$ such that $j=j_{i_0}=j_{i_1}$. Then $\chi(x_j,z)=\chi(a_{i_0},z)=\chi(a_{i_1},z)$. But that contradicts $\chi$ being $2$-bounded.
\end{proof}

\begin{lemma}\label{theideallemma}
Suppose $A\subseteq\omega$ is normal. Let $I$ be an ideal on $\omega$. Let $X\subseteq A$ be polychromatic with $|X|\leq n$. Then if $E(X)\cap A\not\in I$, $$\{a\in A\cap E(X): A\cap E(X\cup a)\in I\}$$ has size at most $n$.
\end{lemma}

\begin{proof}
Suppose for contradiction that $a_0,\ldots a_n$ are distinct members of $A\cap E(X)$ with each $A\cap E(X\cup a_i)$ belonging to $I$. Then by Proposition \ref{union} we have $$A\cap E(X)\subseteq^*A\cap(E(X\cup a_0)\cup\ldots\cup E(X\cup a_n)).$$  But then $A\cap E(X)$ belongs to $I$. Contradiction.
\end{proof}

%It will also be helpful to point out that we can strengthen Lemma \ref{theideallemma} by showing that we can aggregate finitely many notions of largeness into a single one.

%\begin{lemma} \label{themultiideallemma}
%Let $I_0,\ldots I_{k-1}$ be ideals, each including all the finite sets. Suppose $E(X)$ does not belong to any of the $I_j$. Then $$\{a>X:a\in E(X),E(X\cup a)\textrm{ belongs to some }I_j\}$$ has size at most $nk$.
%\end{lemma}

%\begin{proof}
%Otherwise the pigeonhole principle would give, for some $j$, that $$\{a>X:a\in E(X), E(X\cup a)\in I_j\}$$ would have size greater than $n$. And that contradicts Lemma \ref{theideallemma}.
%\end{proof}

The direct proof of the rainbow Ramsey theorem given in \cite{Csima} may be viewed as an application of Lemma \ref{theideallemma} taking $\mathcal{I}$ to be the ideal of finite sets, and serves as a paradigm for our constructions of large polychromatic sets in subsequent subsections. %As an example of how one may apply Lemma \ref{theideallemma} to build large polychromatic sets we use it prove the rainbow Ramsey theorem directly. This proof is essentially in \cite{Csima}.

\subsection{An ultrafilter which is rainbow Ramsey and not rapid}\label{growth}

In this subsection we use $\MA$ to construct a rainbow Ramsey ultrafilter $\mathcal{U}$ which is not rapid. To accomplish this we must be able to build polychromatic sets which are large in the sense that they have enumerating functions which do not grow too fast. That is, we define a function $f:\omega\rightarrow\omega$ and construct $\mathcal{U}$ so that for every $A\in\mathcal{U}$ we have $f\not\leq^*e_A$. Such constructions are not possible in the monochromatic theory; given a function one can always define a $2$-coloring so that any monochromatic set dominates that function.

\begin{prop}\label{buildNrm}
There is a function $\Nrm:\omega^2\rightarrow\omega$ such that the following holds. Suppose $\chi$ is a $2$-bounded coloring and $X$ is normal with $|X|=p$, and suppose $Z\subseteq\omega$ with $Z>X$ and $\Nrm(p,n)\leq |Z|$. There is $Y\subseteq Z$ with $|Y|\geq n$ such that $\chi$ is normal on $X\cup Y$.
\end{prop}

\begin{proof}
For each $a<b\in X$ there is at most one other pair $c<d$ with $\chi(a,b)=\chi(c,d)$ and hence some $z$ from the first ${p \choose 2}+1$ elements of $Z$ gives $X\cup z$ is normal. By iterating this observation we see that we may define $\Nrm$ recursively; $\Nrm(p,n+1)=\Nrm(p,n)+{p+n \choose 2}+1$.
\end{proof}

\begin{prop}\label{buildh}
There is a function $h:\omega^2\rightarrow\omega$ with the following properties. Let $\chi$ be a $2$-bounded coloring and $A\subseteq\omega$ a normal set. Let $\mathcal{I}$ be any ideal on $\omega$, $X\subseteq A$ a polychromatic set with $|X|\leq p$ and $A\cap E(X)\not\in\mathcal{I}$, and $Z\subseteq A\cap E(X)$ with $h(p,n)\leq |Z|$. There is $Y\subseteq Z$ with $|Y|\geq n$ such that $X\cup Y$ is polychromatic and $A\cap E(X\cup Y)\not\in I$.
\end{prop}

\begin{proof}
As in Proposition \ref{buildNrm} this follows by iterating an observation for extending by one point. This time the observation is the claim that given $N_0\in\omega$, if $N_1\geq N_0(p+1)+2p+2$ and $Z$ has size at least $N_1$ then for some $z$ equal to one of the first $2p+1$ members of $Z$ we have $A\cap E(X\cup z)\not\in\mathcal{I}$ and $|Z\cap E(X\cup z)|\geq N_0$.

Let us verify this claim. By Lemma \ref{theideallemma} there are at most $p$ members $z$ of $A\cap E(X)$ with $A\cap E(X\cup z)$ belonging to $I$. Remove these from $Z$ and call the resulting set $Z_0$. Then $Z_0$ has size at least $N_0(p+1)+p+2$ and it is enough to show that one of the first $p+1$ members of $Z_0$ works. Suppose otherwise for contradiction: let $a_0,\ldots a_p$ be the first $p+1$ members of $Z_0$ and assume that each $|Z\cap E(X\cup a_i)|\leq N_0$. Then $$|E(X\cup a_0)\cap Z_0\cup\ldots\cup E(X\cup a_p)\cap Z_0|\leq N_0(p+1).$$ But now we can select a $z$ in $Z_0\subseteq E(X)$ above $a_p$ and not belonging to any of the $E(X\cup a_i)$ and that violates Proposition \ref{union}.
\end{proof}

Now define a function $g$ as follows. For each $n\in\omega$, set $g(1,n)=n+1$. Then recursively define $g$ so that $$g(k+1,n)>h(k,g(k,n)),\Nrm(k,g(k,n)),2\cdot g(k,n).$$ We let $f:\omega\rightarrow\omega$ be a function eventually dominating (for each fixed $k$ and $l$) the map that sends $n$ to $g(k,n)+l$. We say a set is $f$-rapid if $f\leq^*e_A$.

We build $\mathcal{U}$ by constructing a filter $\mathcal{F}$ which consists only of sets which are not $f$-rapid and which contains a polychromatic set for every $2$-bounded coloring; we will then want to extend $\mathcal{F}$ to an ultrafilter consisting only of sets which are not $f$-rapid. To do this it is enough to have $\mathcal{F}\cap\mathcal{I}=\varnothing$ where $\mathcal{I}$ is an ideal on $\omega$ containing all the $f$-rapid sets.

\begin{prop}
Let $\mathcal{I}$ consist of all sets $A\subseteq\omega$ for which there exists $l,k,N\in\omega$ such that $$(\forall n\geq N)|[l,f(n))\cap A|<g(k,n).$$ Then $\mathcal{I}$ is an ideal that contains every $f$-rapid set.
\end{prop}

\begin{proof}
First, if $A$ is $f$-rapid then since $n+1\leq g(1,n)$ taking $l=0$ and $k=1$ witnesses $A\in\mathcal{I}$. Clearly $\mathcal{I}$ is closed under subsets. To see that $\omega\not\in\mathcal{I}$, just notice that $|[l,f(n))\cap\omega|=f(n)-l$ and we defined $f$ so that $g(k,n)+l\leq f(n)$ for sufficiently large $n$. Closure of $\mathcal{I}$ under unions follows from the fact that $2g(k,n)<g(k+1,n)$ for every $n$ and $k$.
\end{proof}

We now build a filter $\mathcal{F}$ disjoint from $\mathcal{I}$ and containing a polychromatic set for each $2$-bounded coloring. We generate $\mathcal{F}$ from a \emph{tower} of sets not in $\mathcal{I}$. Recall that a tower is a sequence $\langle T_i:i<\lambda\rangle$ of subsets of $\omega$ with $i<j$ implying that $T_j\subseteq^* T_i$. 

\begin{prop}\label{richnormal}
If $A\not\in I$ and $\chi:[\omega]^2\rightarrow\omega$ is $2$-bounded then there is a normal $B\subseteq A$ with $B\not\in I$.
\end{prop}

\begin{proof}
It is enough to show that given a finite $X\subseteq\omega$ on which $\chi$ is normal and given $N,k,l$ we may find $n\geq N$ and $Y\subseteq A$ with $X\cup Y$ normal and $$[l,f(n))\cap Y\geq g(k,n)$$ for then $B$ may be constructed by a straightforward induction. Let $p=|X|$. Because $A\not\in\mathcal{I}$ there is some $n\geq N$ such that $|[l,f(n))\cap A|\geq g(k+1,n)>\Nrm(p,g(k,n))$. By Proposition \ref{buildNrm} there is $Y\subseteq[l,f(n))\cap A$ with $|Y|\geq g(k,n)$ and $X\cup Y$ is normal.
\end{proof}

\begin{prop}\label{richpoly}
If $\chi$ a $2$-bounded coloring and $A\not\in I$ is a normal set then there is a set $B\subseteq A$ which is polychromatic and so that $B\not\in\mathcal{I}$.
\end{prop}

\begin{proof}
This is just like Proposition \ref{richnormal} but appealing to Proposition \ref{buildh} instead of Proposition \ref{buildNrm}.
\end{proof}

\begin{lemma}\label{richtower}
Assume $\MA$. If $\langle T_i:i<\lambda\rangle$ is a tower of sets with each $T_i\not\in\mathcal{I}$, there is $T_\lambda\not\in\mathcal{I}$ such that $T_\lambda\subseteq^*T_i$ for all $i<\lambda$.
\end{lemma}

\begin{proof}
The usual forcing notion $\mathbb{P}$ to extend a tower applies; conditions are pairs $\langle s,A\rangle$ where $s$ is a finite subset of $\omega$ and $A$ belongs to the tower. We order by setting $\langle s',A'\rangle\leq\langle s,A\rangle$ exactly when $s'$ is an end extension of $s$, $A'\setminus\max(s')\subseteq A$ and $s'\setminus s\subseteq A$. The dense sets come in two flavors; first for each $T_i$ consider the dense set $C_i$ of conditions $\langle s,A\rangle$ for which $A\setminus\max(s)$ is a subset of $T_i$. Second, for each $l,k,N\in\omega$ take the dense set $D_{l,k,N}$ of conditions $\langle s,A\rangle$ for which there is some $n\geq N$ with $|s\cap [l,f(n))|\geq g(k,n)$. If $G$ is a filter on $\mathbb{P}$ intersecting each $C_i$ and $D_{l,k,N}$ then $T_\lambda$ may be obtained as the union of all the $s$ such that some $\langle s,A\rangle$ belongs to $G$.
\end{proof}

Putting all the ingredients together to construct a $\mathcal{U}$ which is rainbow Ramsey and not rapid is now routine. We enumerate all $2$-bounded colorings $\langle\chi_i:i<2^\omega\rangle$ and construct a tower of sets $\langle T_i:i<2^\omega\rangle$ not in $\mathcal{I}$ such that each $T_i$ is polychromatic for $\chi_i$. Given an initial segment of such a tower, Lemma \ref{richtower} applies to extend it by a single set $A$ and then Lemma \ref{richnormal} followed by Lemma \ref{richpoly} apply to refine this extension to a polychromatic set not in $\mathcal{I}$. With the tower constructed, the filter it generates consists of sets not in $\mathcal{I}$, and this filter can be extended to an ultrafilter $\mathcal{U}$ which is disjoint from $\mathcal{I}$ and is thus not rapid.

\subsection{A weakly selective ultrafilter which is not rainbow Ramsey} \label{wsnotrR}

In this subsection we use $\MA$ to construct a weakly selective ultrafilter $\mathcal{U}$ which is not rainbow Ramsey. First we observe that rainbow Ramsey ultrafilters contain polychromatic sets for colorings with bounds higher than $2$.

\begin{prop}
Suppose that $\mathcal{V}$ is a rainbow Ramsey ultrafilter, $k\in\omega$ and $\chi:[\omega]^2\rightarrow\omega$ is a $k$-bounded coloring. Then there is $A\in\mathcal{U}$ polychromatic for $\chi$.
\end{prop}

\begin{proof}
To simplify the argument we assume $|\chi^{-1}[n]|=k$ for each $n\in\omega$. Let $\{a^n_0,\ldots a^n_{k-1}\}$ enumerate $\chi^{-1}[n]$. For each possible $\{i,j\}\in[k]^2$ let $\chi_{i,j}$ be a two bounded coloring with $\chi(a^n_i)=\chi(a^n_j)$. Then for each $i,j$ we find $A_{i,j}\in\mathcal{V}$ polychromatic for $\chi_{i,j}$. The intersection of all of the finitely many $A_{i,j}$ yields a set in $\mathcal{V}$ which is polychromatic for $\chi$.
\end{proof}

We will define $\mathcal{U}$ to be an ultrafilter on $[\omega]^2$ rather than on $\omega$. We think of $[\omega]^2$ as the set of edges in the complete graph whose set of vertices is $\omega$. Define a coloring $\chi$ by setting $\chi(\{a,b\},\{c,d\})=\{a,b,c,d\}$. Notice that $\chi$ is $4$-bounded. 

Let $\mathcal{I}$ be the collection of $X\subseteq[\omega]^2$ for which there exists an $N$ so that $N\leq |A|$ implies $[A]^2\not\subseteq X$.

\begin{prop} \label{polyIdeal}
The set $\mathcal{I}$ is an ideal containing every set which is polychromatic for $\chi$.
\end{prop}

\begin{proof}
Since $\chi(\{a,b\},\{c,d\})=\chi(\{a,d\},\{b,c\})$ no polychromatic set can contain an $[A]^2$ where $A$ has size at least $4$.

It is clear that $\mathcal{I}$ is closed under subsets. That $\mathcal{I}$ is closed under finite unions follows from the finite monochromatic Ramsey theorem. If $X=Y_0\cup\ldots Y_n$ contains arbitrarily large complete graphs, then so does $Y_i$ for some $i$.
\end{proof}

Let $\mathcal{P}=\{P_n:n\in\omega\}$ be a partition of $[\omega]^2$. We say that $X\subseteq[\omega]^2$ is a weak $\mathcal{P}$-selecter if $X\cap P_n$ is finite for each $n$. To construct our ultrafilter it suffices to build a filter $\mathcal{F}$ disjoint from $\mathcal{I}$ which for each $\mathcal{P}$ either contains some $P_n$ or contains a weak $\mathcal{P}$-selecter for each $\mathcal{P}$. 

\begin{lemma} \label{ramseyTower}
Assume $\MA$. If $\langle T_i:i<\lambda\rangle$ is a tower of subsets of $[\omega]^2$ with each $T_i\not\in\mathcal{I}$, there is $T_\lambda\not\in\mathcal{I}$ such that $T_\lambda\subseteq^*T_i$ for all $i<\lambda$.
\end{lemma}

\begin{proof}
As in Lemma \ref{richtower} we apply $\MA$ to the usual forcing to extend a tower. Our dense sets again come in two flavors; $C_i$ consists of those $\langle s,X\rangle$ with $A\setminus \max(s)$ a subset of $T_i$. For each $N$ we let $D_N$ be those $\langle s,X\rangle$ with some $[A]^2\subseteq s$ with $N\leq |A|$. Given $G\subseteq\mathbb{P}$ a generic filter intersecting each $C_i$ and $D_N$ we may define $T_\lambda$ as the union of all the $s$ such that there exists $X$ with $\langle s,X\rangle\in G$.
\end{proof}

\begin{lemma} \label{ramseySelect}
Suppose $X\not\in\mathcal{I}$ and $\mathcal{P}=\{P_n:n\in\omega\}$ is a partition of $[\omega]^2$. Either some $X\cap P_n\not\in\mathcal{I}$ or there is $Y\subseteq X$ a weak $P$-selecter with $Y\not\in\mathcal{I}$.
\end{lemma}

\begin{proof}
Assume that each $X\cap P_n\in\mathcal{I}$. Then $Y$ may be built in $\omega$-many stages, adding finitely many points at a time. We start with $Y_0=\varnothing$. At stage $N$ we have $Y_N$, with $P_0,\ldots P_n$ all the members of $\mathcal{P}$ with which $Y_N$ has nonempty intersection, and we form $Y_{N+1}=Y_N\cup [B]^2$ where $N\leq |B|$, $[B]^2\subseteq X$ and $[B]^2\cap P_0\cup\ldots\cup P_n=\varnothing$.

We find $B$ as follows. Let $Q=P_0\cup\ldots\cup P_n$. Then $Q\cap X\in\mathcal{I}$ so we may take $M$ so that $Q\cap X$ contains no $[C]^2$ where $|M|\leq C$. Now define a $2$-coloring with domain $X$ by giving elements of $Q$ color 0 and all other members of $X$ color 1. By the monochromatic Ramsey's theorem there is $B\subseteq\omega$ with $M\leq |B|$ and $[B]^2$ monochromatic; then $[B]^2\subseteq X\setminus Q$ as desired. 
\end{proof}

Now a routine recursion of length continuum will yield an ultrafilter $\mathcal{U}$ which is weakly selective but not rainbow Ramsey. Enumerate all partitions of $[\omega]^2$ as $\langle\mathcal{P}^i:i<2^\omega\rangle$ and construct a tower of sets $\langle T_i:i<2^\omega\rangle$ not in $\mathcal{I}$ such that each $T_i$ is either a weak $\mathcal{P}$-selecter or a subset of some $P^i_n$. Given an initial segment of such a tower, Lemma \ref{ramseyTower} applies to extend it by a single set $Y\not\in\mathcal{I}$ and then Lemma \ref{ramseySelect} applies to refine $Y$ to a $Z\subseteq Y$ which is either a weak $\mathcal{P}$-selecter or a subset of some $P^i_n$. With the tower constructed, the filter it generates consists of sets not in $\mathcal{I}$, and this filter can be extended to a weakly selective ultrafilter $\mathcal{U}$ which is disjoint from $\mathcal{I}$ and is thus not rainbow Ramsey.

\subsection{Rainbow Ramsey ultrafilters are nowhere dense.}

In this subsection we show that every rainbow Ramsey ultrafilter is nowhere dense.

\begin{lemma} \label{rRimpliesweaklynwd}
Suppose $S\subseteq\mathbb{R}$ is countable. Then there is a $2$-bounded coloring $\chi:[S]^2\rightarrow\omega$ so that any set $A$ which is polychromatic for $\chi$ is nowhere dense as a subset of $\mathbb{R}$.
\end{lemma}

\begin{proof}
Let $\mathcal{C}$ be the collection of all open intervals in $\mathbb{R}$ with rational endpoints. Fix $\prec$ an ordering of $S$ in ordertype $\omega$. When we specify $\chi(q,r)$ we insist that $q\prec r$. Enumerate $\mathcal{C}=\{c_n:n\in\omega\}$.

Before defining $\chi$ we first define sequences $\{S_n:n\in\omega\}$ and $\{b^{p,q}_n:n\in\omega,\{p,q\}\in[S_n]^2\}$ such that
\begin{enumerate}
\item $S\setminus S_n$ is finite.
\item For each $n\in\omega$ the set $\{b^{p,q}_n:\{p,q\}\in[S_n]^2\}$ is a pairwise disjoint collection of elements of $\mathcal{C}$ each of which is a subset of $c_n$.
\item $b^{p,q}_n\cap b^{r,s}_m\not=\varnothing$ and $\{p,q\}\cap\{r,s\}\not=\varnothing$ implies that $m=n$ (and thus $\{p,q\}=\{r,s\}$ by (2)).
\end{enumerate}

Suppose recursively we have defined $S_i$ and $b_i^{p,q}$ for $i<n$. For each $i\leq n$ we define a set $a_i^n\in\mathcal{C}$ and we do this by recursion. Let $a_0^n=c_n$. If there is some $\{p_i^n,q_i^n\}\in[S_i]^2$ with $a_i^n\cap b_i^{p_i,q_i}\not=\varnothing$ then set $a_{i+1}^n=a_i^n\cap b_i^{p_i^n,q_i^n}$; otherwise set $a_{i+1}^n=a_i^n$. Let $S_n=S\setminus\bigcup_{i<n}\{p_i^n,q_i^n\}$. We take $\{b_n^{p,q}:\{p,q\}\in[S_n]^2\}$ to be a pairwise disjoint collection of members of $\mathcal{C}$ so that each $b_n^{p,q}\subseteq a_n^n$.

That completes the definition of the $S_n$ and the $b_n^{p,q}$. It is easy to see that we have (1) and (2). To see that (3) holds notice that by construction if $i<n$ and $a_n^n\cap b_i^{r,s}\not=\varnothing$ we have $a_n^n\subseteq b_i^{r,s}$ and $\{r,s\}=\{p_i^n,q_i^n\}$. But $p_i^n,q_i^n\not\in S_n$ and $b_n^{p,q}\subseteq a_n^n$.

Now we define $\chi$. For $x$ with $p,q\prec x$ and $x\in S\cap b_n^{p,q}$ set $\chi(p,x)=\chi(q,x)=\{p,q,x\}$. For all other pairs we let $\chi(p,x)=\{p,x\}$. Condition (3) in our construction of $b^{p,q}_n$ guarantees that $\chi$ is well-defined and $2$-bounded.

Suppose $A\subseteq S$ is polychromatic for $\chi$. We want to check that $A$ is nowhere dense. We may as well assume $A$ is infinite. For $c\in\mathcal{C}$ we must find $b\in\mathcal{C}$ with $b\subseteq c$ and $b\cap A=\varnothing$. Fix $c_n\in\mathcal{C}$. Since $A$ is infinite and $S_n$ is coinfinite in $S$ we may find distinct $p,q\in A\cap S_n$. Then $\chi(p,x)=\chi(q,x)$ for $x\in b_n^{p,q}$ and thus $b_n^{p,q}\cap A$ is finite since $A$ is polychromatic. Hence $A$ is empty on a subinterval $b_n^{p,q}$ which itself is a subset of $c_n$.
\end{proof}

Let us temporarily call an ultrafilter {\em weakly nowhere dense} if for every injective $f:\omega\rightarrow\mathbb{Q}$ there is an $A\in\mathcal{U}$ with $f(A)$ nowhere dense. In the terminology of Fla\v{s}kov\'{a} \cite{Flaskova} these are the $\mathcal{I}$-friendly ultrafilters, where $\mathcal{I}$ is the nowhere dense ideal. Lemma \ref{rRimpliesweaklynwd} shows that every rainbow Ramsey ultrafilter is weakly nowhere dense. Thus it only remains for us to establish the following.

\begin{lemma}
Every weakly nowhere dense ultrafilter is nowhere dense.
\end{lemma}

\begin{proof}
We show that given a function $G:\omega\rightarrow\mathbb{R}$ we can find an injective function $F:\omega\rightarrow\mathbb{R}$ such that for any $A\subseteq\omega$, if $F(A)$ is nowhere dense then $G(A)$ is nowhere dense. As before let $\mathcal{C}$ be the collection of all open intervals in $\mathbb{R}$ with rational endpoints.

\medskip

\noindent {\em Claim.} There is $\mathcal{C}'\subseteq\mathcal{C}$ such that for every $c\in\mathcal{C}$ there is a $t\subseteq c$ with $t\in\mathcal{C}'$, and so that each $G(n)$ belongs to finitely many members of $\mathcal{C}$'.

\begin{proof}[Proof of Claim.] Enumerate $\mathcal{C}$ as $\{c_n:n\in\omega\}$. For each $c_n$ let $t_n\subseteq c_n$ be some member of $\mathcal{C}$ that does not contain any of $G(0),\ldots G(n)$. Set $\mathcal{C}'=\{t_n:n\in\omega\}$.
\end{proof}

Now we construct $F$. We do so recursively. Suppose that $F(0),\ldots F(n)$ have already been defined. Let $t_0,\ldots t_m$ be all the members of $\mathcal{C}'$ containing $G(n+1)$. Then $t_0\cap\ldots\cap t_m$ is a nonempty open set. We take $F(n+1)$ to be a member of $t_0\cap\ldots\cap t_m$ different from each of $F(0),\ldots F(n)$.

Obviously $F$ is injective. Notice also that for every $n\in\omega$ and every $t\in\mathcal{C}'$, $F(n)\in t$ whenever $G(n)\in t$. We can use this property to see that $F$ is as desired. For suppose $A\subseteq\omega$ with $F(A)$ nowhere dense. We check that $G(A)$ is nowhere dense. Let $s\in\mathcal{C}$. Since $F(A)$ is nowhere dense there is $u\in\mathcal{C}$ with $u\subseteq s$ so that $F(A)\cap u$ is empty. Take $t\in\mathcal{C}'$ with $t\subseteq u$. Then since $F(A)\cap t$ is empty it follows that $G(A)\cap t$ is empty: if $n\in A$ and $G(n)\in t$ then $F(n)\in t$.
\end{proof}

\subsection{A rainbow Ramsey ultrafilter which is not discrete} \label{discreteUlt}

In this subsection we use $\MA$ to construct  a rainbow Ramsey ultrafilter $\mathcal{U}$ which is not discrete. It is enough to construct $\mathcal{U}$ an ultrafilter on $\mathbb{Q}$ which contains no discrete subset of $\mathbb{Q}$ but which contains a polychromatic set for every $2$-bounded coloring on $\mathbb{Q}$. 

\begin{defn}
Let $A\subseteq\mathbb{Q}$. We define $L^k(A)\subseteq\mathbb{Q}$ by induction on $k\in\omega$.
\begin{enumerate}
\item $L^0(A)=A$.
\item $L^{k+1}(A)$ is the set of $a\in A$ which are limit points of $L^k(A)$.
\end{enumerate}
If there exists $k$ such that $L^{k}(A)=\varnothing$ we set $\CB(A)$ equal to the least such $k$. Otherwise we say $\CB(A)\geq\omega$.
\end{defn}

For $A$ without a perfect kernel $\CB(A)$ is just the usual Cantor-Bendixson rank for finitely ranked sets. Let $\mathcal{I}$ be the collection of all $A\subseteq\Q$ with $\CB(A)<\omega$. The following proposition is well-known but we include a proof since it uses ideas we will need later in the more delicate situation of Proposition \ref{maintainRobust}.

\begin{prop}\label{CBideal}
The set $\mathcal{I}$ contains every discrete subset of $\Q$ and is an ideal. In fact if $\CB(A\cup B)>k+l$ then $\CB(A)>k$ or $\CB(B)>l$.
\end{prop}

\begin{proof}
Every discrete subset of $\Q$ belongs to $\mathcal{I}$ since a set is discrete exactly when $\CB(A)<2$.

To show that $\mathcal{I}$ is an ideal, fix $A,B\subseteq\mathbb{Q}$. We show that if $L^{k+l}(A\cup B)\not=\varnothing$ then either $L^k(A)\not=\varnothing$ or $L^l(B)\not=\varnothing$. 

For our purposes a tree $(T,<_T)$ is a partially ordered set so that for each node $x\in T$ the set of predecessors of $x$ is well-ordered by $<$. Let $p_T(x)$ be the set of predecessors of $x$. For each $x\in T$ we let the height of $x$, written $\hgt(x)$, be equal to the order-type of $p_T(x)$. The height of $T$, written $\hgt_T(T)$, is the maximum of the heights of its elements. The $k$th level of $T$ is all $x\in T$ with $\hgt_T(x)=k$.

Now let $\mathcal{T}$ be the collection of all trees $T$ with the following properties:
\begin{enumerate}
\item $T$ has finite height.
\item $T\subseteq A\cup B$.
\item Every $x\in T$ with $\hgt_T(x)<\hgt(T)$ has infinitely many successors. These can be enumerated $\{y_n:n\in\omega\}$ where $\lim_{n<\omega}y_n=x$.
\item \label{uniform} For each $k\leq\hgt(T)$ either every $x\in T$ with $\hgt_T(x)=k$ belongs to $A$, or every $x\in T$ with $\hgt_T(x)=k$ belongs to $B$.
\end{enumerate}

\medskip

If $x\in L^k(A\cup B)$ then by induction on $k$ it may be shown that there is $T\in\mathcal{T}$ with root $x$ and $\hgt_T(T)=k$. %This is proven by induction on $k$. If $k=0$ this is easy. Suppose we want to handle $k+1$; since $x\in L^{k+1}(A\cup B)$ there is some sequence $y_0,\ldots y_n,\ldots$ of members of $L^k(A\cup B)$ with $\lim_{n<\omega}y_n=x$. By thinning the sequence appropriately we may assume each $y_n$ belongs to $A$ (the $B$ case is symmetric). Now by induction, for each $n<\omega$ there is a $T_n\in\mathcal{T}$ with root $y_n$ and $\KB_{T_n}(y_n)=k$. To form $T$ we cannot just place all the $T_n$ above $x$ as we may not get the desired uniformity as in clause \ref{uniform}. But we can thin again. For each $T_n$ some sequence like $AABBA$ represents the uniformity of each level. There are infinitely many $T_n$ so one of those sequences gets repeated infinitely many times and we may place the corresponding $T_n$s above $x$ to get our $T$.
Thus if there is some $x\in L^{k+l}(A\cup B)$ then there is a $T\in\mathcal{T}$ with root $x$ and $\hgt_T(T)=k+l$. Then $T$ has $k+l+1$ levels so either $k+1$ levels are subsets of $A$ or $l+1$ levels are subsets of $B$. It is easy to see if a node in $A$ has $k$ levels above which are subsets of $A$ then that node belongs to $L^k(A)$. Similarly for $B$ and $l$. 
\end{proof}

There is an unfortunate complication in the argument to come. Unlike our constructions in subsections \ref{growth} and \ref{wsnotrR} we will not be able to generate $\mathcal{U}$ from a tower; there are countable length towers of sets not in $\mathcal{I}$ which cannot be extended by a set not in $\mathcal{I}$.

The key lemma in the construction of $\mathcal{U}$ is the following.

\begin{lemma} \label{discreteMainStep}
Assume $\MA$. Let $\mathcal{S}$ be a filter base with $\mathcal{S}\cap\mathcal{I}=\varnothing$ and $|\mathcal{S}|<2^\omega$ and let $\chi$ be a $2$-bounded coloring on $\mathbb{Q}$. There is a polychromatic $B\subseteq\mathbb{Q}$ such that $B\cap S\not\in\mathcal{I}$ for every $S\in\mathcal{S}$.
\end{lemma}

Using Lemma \ref{discreteMainStep} to construct an appropriate $\mathcal{U}$ is a routine recursion of length continuum. So we turn to proving Lemma \ref{discreteMainStep}. Fix $\mathcal{S}$ a filter base with $\mathcal{S}\cap\mathcal{I}=\varnothing$ and $|\mathcal{S}|<2^\omega$ and fix $\chi:[\mathbb{Q}]^2\rightarrow\omega$ a $2$-bounded coloring. Throughout the rest of this section we assume $\MA$. For $\epsilon>0$ and $a\in\mathbb{Q}$ we let $N_\epsilon(a)$ denote the $\epsilon$-neighborhood around $a$. Fix a well-ordering $\prec$ on $\mathbb{Q}$ of ordertype $\omega$; our definition of normal for subsets $\mathbb{Q}$ is the same as Definition \ref{normDef}, but using $\prec$ instead of $<$.  

\begin{prop}\label{CBnormy}
There is a normal $A\subseteq\mathbb{Q}$ with $A\cap \mathcal{S}\not\in\mathcal{I}$ for each $S\in\mathcal{S}$.
\end{prop}

\begin{proof}
We apply $\MA$ to the partial order of finite normal subsets of $\mathbb{Q}$ ordered by $X_1\leq X_0$ if $X_0\subseteq X_1$. The point is to arrange the dense sets so that for each $k\in\omega$ and $S\in\mathcal{S}$ we eventually add a member of $L^k(S)$, and once we have added some $a\in L^{k+1}(S)$ we add for each rational $\epsilon>0$ a member of $L^k(S)\cap N_\epsilon(a)$. The density of the sets follows from the fact that each $S\not\in\mathcal{I}$ and the fact that if $X\subseteq\mathbb{Q}$ is finite and normal then $X\cup\{a\}$ is normal for all but finitely many $a\in\mathbb{Q}$. 
\end{proof}

Now fix $A$ as in Proposition \ref{CBnormy}. We want to build a polychromatic $B\subseteq A$ with each $B\cap S\not\in\mathcal{I}$. We will build $B$ by finite approximations, which we denote by $X$, and when we add a proposed limit point $b$ to $B$ we have to make sure that $b$ is a limit point not only of $E(X)$ but also of $E(X\cup b)$. Hence the following definition. 

\begin{defin}
Let $X\subseteq A$ be finite and polychromatic and let $S\in\mathcal{S}$. We define $L^k_\pol(X,S)$ by induction on $k\in\omega$.
\begin{enumerate}
\item $L^0_\pol(X,S)=E(X)\cap S$.
\item $L^{k+1}_\pol(X,S)$ is the set of $a\in E(X)\cap S$ which are limit points of $L^k_\pol(X\cup a,S)$.
\end{enumerate}
If there exists $k$ such that $L^k_\pol(X,S)=\varnothing$ we set $\CB_\pol(X,S)$ equal to the least such $k$. Otherwise we say $\CB_\pol(X,S)\geq\omega$.
\end{defin}

We prove the analogue of Proposition \ref{union}. Define $h:\omega^2\rightarrow\omega$ by recursion on the first coordinate. Take $h(0,n)=n$ and $h(k+1,n)=n+1+h(k,n+1)$.

\begin{prop} \label{limitUnion}
Suppose $n=|X|$ and $a_0,\ldots a_{h(k,n)}$ are distinct members of $E(X)\cap S$. Then $$L^k_\pol(X,S)\subseteq^*L^k_\pol(X\cup a_0,S)\cup\ldots\cup L^k_\pol(X\cup a_{h(k,n)},S).$$
\end{prop}

\begin{proof}
The base case $k=0$ is just Proposition \ref{union}.

For the successor case let $y\in L_\pol^{k+1}(X,S)$. By Proposition \ref{union} there are at most $n+1$ choices of $i$ with $y\not\in E(X\cup a_i)$ or equivalently $a_i\not\in E(X\cup y)$. Hence by relabeling we may assume that $a_i\in E(X\cup y)$ for $i\leq h(k,n+1)$. By definition $y$ is a limit point of $L^k_\pol(X\cup y,S)$ and by induction we have $$L^k_\pol(X\cup y,S)\subseteq^* L^k_\pol(X\cup y\cup a_0,S)\cup\ldots\cup L^k_\pol(X\cup y\cup a_{h(k,n+1)},S).$$ Thus there is $i$ such that $y$ is a limit point of $L^k_\pol(X\cup y\cup a_i,S)$. Since $y\in E(X\cup a_i)\cap S$ that gives $y\in L^{k+1}_\pol(X\cup a_i,S)$.
\end{proof}

\begin{prop} \label{startsRobust}
For each $S\in\mathcal{S}$ we have $\CB_\pol(\varnothing,S)\geq\omega$.
\end{prop}

\begin{proof}
We must prove that $L^k_\pol(X,S)\not=\varnothing$ for each $k<\omega$. Define $v:\omega^2\rightarrow\omega$ by recursion on the first coordinate. Take $v(0,n)=1$ and $v(k+1,n)=v(k,n)+h(k,n)+2$.

To prove the proposition, we prove the following more general fact by induction on $k$. For $X\subseteq A$ with $|X|=n$ and $U$ an open subset of $\Q$,
$$(*)\textrm{ if }L^{v(k,n)+1}(E(X)\cap S)\cap U\not=\varnothing\textrm{ then }L^{k}_\pol(X,S)\cap U\textrm{ is infinite}.$$ Since $\CB(E(\varnothing)\cap S)\geq\omega$, this statement yields the proposition when used with $X=\varnothing$ and $U=\mathbb{Q}$.

The base case $k=0$ is trivial.

For the successor step $k+1$ suppose that $L^{k+1}_\pol(X,S)\cap U$ is finite, yet $L^{v(k+1,n)+1}(E(X)\cap S)\cap U$ is not empty. We define sequences $\{y_i:i\leq h(k,n)\}$, $\{\epsilon_i:i\leq h(k,n)\}$ by recursion so that
\begin{enumerate}
\item $N_{\epsilon_0}(y_0)\subseteq U$
\item $y_i\in L^{v(k+1,n)-i}(E(X)\cap S)$
\item $N_{\epsilon_i}(y_i)\cap L^k_\pol(X\cup y_i,S)\subseteq\{y_i\}$
\item $N_{\epsilon_{i+1}}(y_{i+1})\subseteq N_{\epsilon_i}(y_i)$.
\end{enumerate}

Start by fixing some $y$ a member of $L^{v(k+1,n)+1}(E(X)\cap S)\cap U$. By definition $y$ is a limit point of $L^{v(k+1,n)}(E(X)\cap S)$; since $y\in U$ there must be infinitely many members of $L^{v(k+1,n)}(E(X)\cap S)\cap U$; one of them does not belong to $L^{k+1}_\pol(X,S)$, take this to be $y_0$. Since $y_0\not\in L^{k+1}_\pol(X,S)$ we may select some small $\epsilon_0$ satisfying (1) and (3). The contruction of the rest of the sequence follows suit and we obtain $y_{i+1}$ from $y_i$ in a manner similar to how we obtained $y_0$ from $y$.

Now take $y=y_{h(k,n)}$ and $\epsilon=\epsilon_{h(k,n)}$. By (3) and (4) we have that $N_\epsilon(y)\cap L_\pol^k(X\cup y_i,S)$ is finite for $i\leq h(k,n)$. Also $y\in N_\epsilon(y)\cap L^{v(k,n)+1}(E(X)\cap S)$ so that by induction $N_\epsilon(y)\cap L^k_\pol(X,S)$ is infinite. And yet by Proposition \ref{limitUnion} we have $$N_\epsilon(y)\cap L^k_\pol(X,S)\subseteq^* N_\epsilon(y)\cap (L^k_\pol(X\cup y_0,S)\cup\ldots\cup L^k_\pol(X\cup y_{h(k,n)},S)).$$ This is a contradiction because the right hand side is supposedly finite.
\end{proof}

\begin{prop} \label{maintainRobust}
Suppose that $X\subseteq A$ with $|X|\leq n$ and $\CB_\pol(X,S)\geq\omega$ for each $S\in\mathcal{S}$. Then there are at most $n$ elements $a$ in $E(X)$ such that $\CB_\pol(X\cup a,S)<\omega$ for some $S\in\mathcal{S}$.
\end{prop}

\begin{proof}
Suppose for contradiction that $a_i\in E(X)$ and $S_i\in\mathcal{S}$ for $i\leq n$ with $\CB_\pol(X\cup a_i,S_i)<\omega$. Set $S=\bigcap_{i<n}S_i$. Then $\CB_\pol(X,S)\geq\omega$ while $\CB_\pol(X\cup a_i,S)<\omega$ for each $i\leq n$. By Proposition \ref{union} we have $$E(X)\cap S\subseteq^* (E(X\cup a_0)\cup\ldots\cup E(X\cup a_n))\cap S.$$ We will use this obtain a contradiction by showing that for each $l\in\omega$ there is some $i\leq n$ so that $L_\poly^l(X\cup a_i,S)$ is not empty. So fix $l\in\omega$.

Let $\mathcal{T}$ be the collection of all trees $T$ with the following properties:
\begin{enumerate}
\item $T$ has finite height.
\item $x<_Ty$ in $T$ implies $x\prec y$.
\item $T\subseteq (E(X\cup a_0)\cup\ldots E(X\cup a_n))\cap S$.
\item Every $x\in T$ with $\hgt_T(x)<\hgt(T)$ has infinitely many successors. These can enumerated be as $\{y_n:n\in\omega\}$ where $\lim_{n<\omega}y_n=x$.
\item For each $k\leq\hgt(T)$ for some $i\leq n$ we have that all $x$ in $T$ with $\hgt(x)=k$ belongs to $E(X\cup a_i)$.
\item \label{colorpart} $x\in E(X\cup p_T(x))$ for each $x\in T$.
\end{enumerate}

If $x\in L_\pol^k(X,S)$ and does not belong to the finite set $F$ it can be shown by inudction on $k$ that there is a $T\in\mathcal{T}$ with root $x$ and $\hgt(T)=k$. Since $\CB_\pol(X,S)\geq\omega$ it follows that we may find $T\in\mathcal{T}$ with $\hgt(T)$ arbitrarily large.

Let $\mathcal{T}_i$ be all trees $T'\subseteq E(X\cup a_i)\cap S$ satisfying clauses 1,2, and 4 of the definition of $\mathcal{T}$ with the additional property that if $x\in T'$ then $x\in E(X\cup a_i\cup p_{T'}(x))$. If $x$ is the root of some $T'\in\mathcal{T}_i$ with $\hgt(T')=l$ then $x\in L^l_\pol(X\cup a_i,S)$, so we just need to find such a tree.

Let $T\in\mathcal{T}$ with $\hgt(T)> 2l(n+1)$. For some $i\leq n$ there are $2l+1$ levels of the tree with every member of that level belonging to $E(X\cup a_i)$. Let $T_0$ be the subtree of $T$ consisting of just those levels. Then $\hgt(T_0)=2l$ and $T_0$ satisfies all the requirements of the definition of membership in $\mathcal{T}_i$ except possibly one: while each $x\in T_0$ belongs to both $E(X\cup a_i)$ and $E(X\cup p_{T_0}(x))$ it may be that $x$ does not belong to $E(X\cup a_i\cup p_{T_0}(x))$.

We define by recursion a sequence of trees $\{T_k:k\leq l\}$ so that 
\begin{enumerate}
\item[(i)] $\hgt(T_{k+1})=\hgt(T_k)-1=2l-k$  
\item[(ii)] For $x\in T_k$ with $\hgt_{T_k}(x)\geq\hgt(T_k)-k$ we have $x\in E(X\cup a_i\cup p_{T_k}(x))$. 
\end{enumerate}

Given the sequence we may take $T'$ to be $T_l$. So let us describe the recursion. Say $T_k$ is given with $k<l$. For $j<\hgt_{T_k}(x)$ let $p_{T_k}(j,x)$ denote the predecessor $y$ of $x$ with $\hgt_T(y)=j$. If $x\in T_k$ and $x\not\in E(X\cup a_i\cup p_{T_k}(x))$ then by normality of $A$ and the fact that $x$ belongs to both $E(X\cup a_i)$ and $(E(X)\cup p_{T_0}(x))$ there is some $j<\hgt_{T_0}(x)$ with $\chi(p(j,x),x)=\chi(a,x)$. Since $\chi$ is $2$-bounded there is at most one such $j$. For each $x\in T_k$ with $\hgt_{T_k}(x)=\hgt(T_k)-k-1$ let $b(x)$ equal such a $j$ if it exists. We may thin $T_k$ so that $b(x)$ is the same $j$ for all such $x$; then form $T_{k+1}$ by removing level $j$ from $T_k$.
\end{proof}

\begin{prop} \label{maintainLimitPoint}
Suppose that $c\in A$, $X\subseteq A$ is finite and $S\in\mathcal{S}$. If $c$ is a limit point of $L^k_\pol(X,S)$ then for all but finitely many $a\in E(X)\cap S$ we have that $c$ is a limit point of $L^k_\pol(X\cup a,S)$.
\end{prop}

\begin{proof}
This is immediate from Proposition \ref{limitUnion}.
\end{proof}

\begin{proof}[Proof of Lemma \ref{discreteMainStep}]
Take $\mathbb{P}$ to the notion of forcing consisting of conditions $\langle X,f\rangle$ where
\begin{enumerate}
\item $X\subseteq A$ is finite and polychromatic with $\CB_\pol(X,S)\geq\omega$ for each $S\in\mathcal{S}$.
\item $f$ is a finite partial function from $\mathcal{S}\times\omega\times\omega$ into $X$ so that if $f(S,n,k)=c$ then $c$ is a limit point of $L^k_\pol(X,S)$.
\end{enumerate}

(We could also prove this lemma using the same forcing notion but without the commitments $f$, but including them will help keep the argument organized.) We order $\mathbb{P}$ by inclusion: $\langle X',f'\rangle\leq\langle X,f\rangle$ if and only $X'\supseteq X$ and $f'\supseteq f$. That $\mathbb{P}$ is nonempty follows from Proposition \ref{startsRobust}. A simple $\Delta$-system argument establishes that $\mathbb{P}$ is ccc.

Let $D_{S,n,k}$ be the collection of conditions $\langle X,f\rangle$ with $(S,n,k)\in\dom(f)$. To check density let $\langle X,f\rangle\in\mathbb{P}$. The set $L^{k+1}_\pol(X,S)$ is infinite. Together Propositions \ref{maintainRobust} and \ref{maintainLimitPoint} imply that for all but finitely many $c\in L^{k+1}_\pol(X,S)$, the pair $\langle X\cup\{c\},f\cup\{\langle(S,n,k),c\rangle\}$ is a condition. It clearly belongs to $D_{S,n,k}$.

For rational $\epsilon>0$ let $E_{S,n,k,\epsilon}$ be the collection of conditions $\langle X,f\rangle$ with $f(S,n,k+1)=c$ where for some $m$ we have $f(S,m,k)=d$ and $d\in N_\epsilon(c)$. To check density let $\langle X,f\rangle\in\mathbb{P}$ and without loss of generality assume $\langle X,f\rangle\in D_{S,n,k+1}$ and $f(S,n,k+1)=c$. Then $N_\epsilon(c)\cap L^{k+1}_\pol(X,S)$ is infinite and together Propositions \ref{maintainRobust} and \ref{maintainLimitPoint} imply that for all but finitely many $d\in N_\epsilon(c)\cap L^{k+1}_\pol(X,S)$ the pair $\langle X\cup\{c\},f\cup\{\langle (S,m,k),d\rangle\}$ is a condition. It clearly belongs to $E_{S,n,k,\epsilon}$. 

Finally let $D'_{S,n,\epsilon}$ be the collection of conditions $\langle X,f\rangle$ with $f(S,n,0)=c$ where we have some $d\in X\cap N_\epsilon(c)$. Density of $D'_{S,n,\epsilon}$ can be checked similarly to the above.

Using $\MA$ let $G\subseteq\P$ be a filter intersecting the dense sets described above. Let $B=\bigcup\{X:\langle X,f\rangle\in G\}$ and let $F=\bigcup\{f:\langle X,f\rangle\in G\}$. Then $B$ is polychromatic and an induction on $k$ shows that each $F(S,n,k)$ belongs to $L^{k+1}(B)$.
\end{proof}

\section{Polychromatic Ramsey theory and cardinal characteristics of the continuum} \label{RRTnCC}

In this short final section we give a few well-known cardinal characteristics characterizations with the flavor of polychromatic Ramsey theory. The colorings we use here will be unary.

\begin{defin}
Let $\mathcal{F}\subseteq\omega^\omega$. We let $\Par(\mathcal{F})$ denote the least size of a family $\mathcal{G}\subseteq\mathcal{F}$ for which for every $X\in[\omega]^\omega$ there is $f\in\mathcal{G}$ so that $f$ is neither eventually constant nor eventually injective on $X$.
\begin{enumerate}
\item $\Par_{1c}=\Par(\omega^\omega)$.
\item $\Par_c=\Par(2^\omega)$
\item $\Par_1=\Par(\mathcal{F})$ where $\mathcal{F}$ consists on all finite-to-one functions.
\item $\Par_\bdd=\Par(\mathcal{F})$ where $\mathcal{F}$ consists of all $f$ with each $|f^{-1}(n)|\leq 2$.
\end{enumerate}
\end{defin}

Our notation is consistent with that of Blass \cite{Bl3} who introduced $\Par_{1c}$. The cardinal $\Par_c$ is just the splitting number $\mathfrak{s}$. Let us note that Galvin's trick applied to the unary $2$-bounded colorings corresponding to $\Par_\bdd$ yields the inequality $\Par_c\leq \Par_\bdd$. 

We also introduce notation for the dual characteristics.

\begin{defn}
Let $\mathcal{F}\subseteq\omega^\omega$. We let $\Hom(\mathcal{F})$ denote the least size of a $\mathcal{X}\subseteq[\omega]^\omega$ so that for every $f\in\mathcal{F}$ there is some $X\in\mathcal{X}$ so that $f$ is either eventually constant or eventually injective on $X$.
\begin{enumerate}
\item $\Hom_{1c}=\Hom(\omega^\omega)$.
\item $\Hom_c=\Hom(2^\omega)$
\item $\Hom_1=\Hom(\mathcal{F})$ where $\mathcal{F}$ consists on all finite-to-one functions.
\item $\Hom_\bdd=\Hom(\mathcal{F})$ where $\mathcal{F}$ consists of all $f$ with each $|f^{-1}(n)|\leq 2$.
\end{enumerate}
\end{defn}

\begin{prop}
$\Par_1=\mathfrak{b}$, and dually $\Hom_1=\mathfrak{d}$.
\end{prop}

\begin{proof}
Given $f\in\omega^\omega$ strictly increasing let $g$ be some finite to one function which is constant on each interval $[f(2n),f(2n+2))$. Then if $X\in[\omega]^\omega$ and $g_f$ is injective on a cofinite subset of $X$ it follows that $f\leq^*e_X$. This shows $\Par_1\leq\mathfrak{b}$. The dual argument shows that $\mathfrak{d}\leq\Hom_1$.

For each strictly increasing $f\in\omega^\omega$ let $X_f\in[\omega]^\omega$ be the set $\{f^n(0):n\in\omega\}$. For each finite-to-one function $g\in\omega^\omega$ let $h_g\in\omega^\omega$ be such that if $l\geq h_g(n)$ then $g(l)\not\in\{g(0),\ldots g(n)\}$. Suppose $h_g\leq^* f$ and take $N$ for which $f(n)\geq h_g(n)$ for all $n\geq N$. Then $g$ is injective on $X_f\setminus N$. This shows $\Hom_1\leq\mathfrak{d}$. The dual argument shows that $\mathfrak{d}\leq\Par_1$.
%if f(n)\geq N then f(f(n))\geq h_g(f(n)) so g(f(f(n)) is not equal to any of g(0),...g(f(n)). Similarly for f^k(n)
\end{proof}

\begin{thm}
$\Par_\bdd=\textbf{non}(\mathcal{M})$, and dually $\Hom_\bdd=\textbf{cov}(\mathcal{M})$.
\end{thm}

\begin{proof}
First we prove $\Par_\bdd\leq\textbf{non}(\mathcal{M})$. The collection $\mathcal{F}\subseteq\omega^\omega$ of two-to-one functions is closed as a subset of Baire space and thus may be regarded as a Polish space in its own right. Thus if $\textbf{non}(\mathcal{M})<\Par_\bdd$ there exists some nonmeager $A\subseteq\mathcal{F}$ with cardinality strictly less than $\Par_\bdd$. By definition of $\Par_\bdd$ there exists an infinite $X\subseteq\omega$ for which $$A\subseteq\{f\in\mathcal{F}:(\exists N)(\forall n,m\in X)n,m\geq N\rightarrow f(n)\not=f(m)\}.$$ But this latter set is meager, contradiction.

The dual inequality, $\textbf{cov}(\mathcal{M})\leq\Hom_\bdd$, can be obtained using a dual argument. Alternatively one may notice that by its definition $\Hom_\bdd$ is a $\bSigma^0_2$ characteristic and apply Proposition 3 and Theorem 5 of \cite{Bl3}.
 
Next we prove $\Hom_\bdd\leq\textbf{cov}(\mathcal{M})$. We use Bartoszy\'{n}ski's characterization of $\textbf{cov}(\mathcal{M})$ in terms of slaloms. A slalom is a function $\phi$ with domain $\omega$ so that each $\phi(n)\subseteq\omega$ with $|\phi(n)|\leq n$. Let $\mathcal{C}$ denote the set of slaloms. Then $\textbf{cov}(\mathcal{M})$ is the least size of a family $\mathcal{F}\subseteq\omega^\omega$ such that for every $\phi\in\mathcal{C}$ there is some $f\in\mathcal{F}$ so that for all but finitely many $n$ we have $f(n)\not\in\phi(n)$. For a proof of this characterization, see Lemma 2.4.2 in \cite{BJ}. Let us point out that for the purposes of this characterization the requirement that $|\phi(n)|\leq n$ is unnecessary; instead of the identity function we may use any function $h\in\omega^\omega$ with values going to infinity and require $|\phi(n)|\leq h(n)$ instead.

We start by massaging Bartoszy\'{n}ski's characterization slightly, and show that we may take $\mathcal{F}$ to consist of strictly increasing functions.
%So we show that $$\|(\mathcal{C},\omega^\omega,\forall^\infty\not\ni)\|=\|(\mathcal{C},\omega^{\uparrow\omega},\forall^\infty\not\ni)\|.$$ That $\leq$ holds is clear. To show $\geq$ we exhibit a morphism, and use slaloms $\phi$ with $|\phi(n)|\leq n$ on the right and slaloms $\phi$ with $|\phi(n)|\leq (1+2+\cdots+n+n+1)$ on the left.
For each $f\in\omega^\omega$ we associate a strictly increasing $g_f\in\omega^\omega$ as follows. If $f$ is finite-to-one, we fix some $X=\{x_n:n\in\omega\}$ (enumerated in increasing order) on which $f$ is strictly increasing and set $g_f(n)=f(x_n)$. Otherwise we take $g_f$ to be the identify function (or something equally arbitrary). We claim that if $\mathcal{F}\subseteq\omega^\omega$ is such that $(\forall \phi\in\mathcal{C})(\exists f\in\mathcal{F})(\forall^\infty n)f(n)\not\in\phi(n)$ then the family $\{g_f:f\in\mathcal{F}\}$ has the same property (with respect to a class of slaloms with a larger bound).

Given $\phi\in\mathcal{C}$ associate the function $\psi_\phi$ defined by $$\psi_\phi(n)=\phi(0)\cup\ldots\phi(n)\cup\{0,\ldots n\}.$$ It is enough to show that if $(\forall^\infty n) f(n)\not\in\psi_\phi(n)$ then $(\forall^\infty n) g_f(n)\not\in\phi(n)$. If $(\forall^\infty n)f(n)\not\in\psi_\phi(n)$ $f$ certainly can only take each value finitely often. Further for sufficiently large $n$ we have $g_f(n)=f(x_n)\not\in\psi_\phi(n)$. Then for such $n$, $$g_f(n)\not\in \phi(0)\cup\ldots\cup\phi(x_n).$$ Since $n\leq x_n$ we have $g_f(n)\not\in\phi(n)$, as desired.

We now use the massaged characterization to finish the theorem. Given strictly increasing $f\in\omega^\omega$ associate $A_f\in[\omega]^\omega$ given by $A_f=\{f^n(0):n\in\omega\}$. To each two-to-one $g\in\omega^\omega$ we associate a slalom $\phi_g$ defined as follows. First define $h_g\in\omega^\omega$ by setting $h_g(n)=m$ where $m\not=n$ is the unique $m$ such that $g(n)=g(m)$ if such $m$ exists, and set $h_g(n)=n$ if there is no such $m$. Then define $\phi_g$ by $$\phi_g(n)=\{h_g(0),\ldots h_g(n)\}.$$

We verify that if $(\forall^\infty n)f(n)\not\in\phi_g(n)$ then $g$ is injective on a cofinite subset of $A_f$. Say $N$ is such that $f(n)\not\in\phi_g(n)$ for $n\geq N$. Then $g$ is injective on $A_f\setminus N$. Otherwise for some $k,m\in\omega$ with $f^k(0)\geq N$ we would have $g(f^k(0))=g(f^{k+m+1}(0))$. Then $h_g(f^k(0))=f^{k+m+1}(0)$. Because $f$ is increasing we have $h_g(f^k(0))\in\phi_g(f^{k+m}(0))$. Thus $f(f^{k+m}(0))\in\phi_g(f^{k+m}(0))$, contradicting $f^{k+m}(0)\geq N$.
\end{proof}

\bibliographystyle{alpha}
\bibliography{RainbowRamsey}

%\begin{thebibliography}{9}

%\bibitem{Je1} Jech, T., \emph{The Axiom of Choice}, North Holland, 1973. MR 53:139

%\bibitem{Kl} Kleinberg, E.M. The Independence of Ramsey's Theorem. Journal of Symbolic Logic 34, 205–206 , 1969.

%\bibitem{BJ} Tomek Bartoszy\'{n}ski and Haim Judah, \emph{Set Theory: on the structure of the real line}, A.K. Peters, 1995. CMP 96:01.

%\bibitem{Bl} Andreas Blass. Combinatorial cardinal characteristics of the continuum. In: Foreman, M., Magidor, M., Kanamori, A. (eds.) Handbook of Set Theory.

%\bibitem{Bl2} Andreas Blass. Ramsey's theorem in the hierarchy of choice principles, Journal of Symbolic Logic 42, 387-390, 1977.

%\bibitem{Bl3} Andreas Blass. Simple cardinal

%\bibitem{CM}
%B. F. Csima and J. R. Mileti. The Strength of the Rainbow Ramsey Theorem. \emph{J. Symbolic Logic}, to appear.

%\end{thebibliography}

\end{document}